\documentclass[oneside,english]{amsart}
\usepackage[T1]{fontenc}
\usepackage[latin9]{inputenc}
\usepackage{amsbsy}
\usepackage{amstext}
\usepackage{amsthm}
\usepackage{amssymb}
\usepackage[nocompress]{cite}

\makeatletter
\numberwithin{equation}{section}
\numberwithin{figure}{section}
\theoremstyle{plain}
\newtheorem{thm}{\protect\theoremname}[section]
\theoremstyle{definition}
\newtheorem{defn}[thm]{\protect\definitionname}
\theoremstyle{definition}
\newtheorem{example}[thm]{\protect\examplename}
\theoremstyle{plain}
\newtheorem{prop}[thm]{\protect\propositionname}

\makeatother

\usepackage{babel}
\providecommand{\definitionname}{Definition}
\providecommand{\examplename}{Example}
\providecommand{\propositionname}{Proposition}
\providecommand{\theoremname}{Theorem}

\begin{document} \sloppy
\title[High dimensional Hoffman bound and extremal combinatorics]{High dimensional Hoffman bound and applications in extremal combinatorics}
\author[Y. Filmus]{Yuval Filmus}
\author[K. Golubev]{Konstantin Golubev}
\author[N. Lifshitz]{Noam Lifshitz}
\address[Y. Filmus]{yuvalfi@cs.technion.ac.il, Technion -- Israel Institute of Technology, Haifa, Israel.}
\address[K. Golubev]{golubevk@ethz.ch, D-MATH, ETH Zurich, Switzerland.}
\address[N. Lifshitz]{noam.lifshitz@gmail.com, Einstein Institute of Mathematics, Hebrew University, Jerusalem, Israel.}
\date{\today}
\begin{abstract}
The $n$-th tensor power of a graph with vertex set $V$ is the graph
on the vertex set $V^{n}$, where two vertices are connected by an
edge if they are connected in each coordinate. The problem of studying
independent sets in tensor powers of graphs is central in combinatorics,
and its study involves a beautiful combination of analytical and combinatorial
techniques. One powerful method for upper-bounding the largest independent
set in a graph is the Hoffman bound, which gives an upper bound on
the largest independent set of a graph in terms of its eigenvalues.
It is easily seen that the Hoffman bound is sharp on the tensor power
of a graph whenever it is sharp for the original graph.

In this paper we introduce the related problem of upper-bounding independent
sets in tensor powers of hypergraphs. We show that many of the prominent
open problems in extremal combinatorics, such as the Turán problem
for (hyper-)graphs, can be encoded as special cases of this problem.
We also give a new generalization of the Hoffman bound for hypergraphs
which is sharp for the tensor power of a hypergraph whenever it is
sharp for the original hypergraph.

As an application of our Hoffman bound, we make progress on the following
problem of Frankl from 1990. An extended triangle in a family of sets
is a triplet $\left\{ A,B,C\right\} \subseteq\binom{[n]}{2k}$ such
that each element of $[n]$ belongs either to none of the sets $\left\{ A,B,C\right\} $
or to exactly two of them. Frankl asked how large can a family $\mathcal{F}\subseteq\binom{\left[n\right]}{2k}$
be if it does not contain a triangle. We show that if $\frac{1}{2}n\le2k\le\frac{2}{3}n,$
then the extremal family is the \emph{star, }i.e. the family of all
sets that contains a given element. This covers the entire range in
which the star is extremal. As another application, we provide spectral proofs for Mantel's theorem on triangle-free graphs and for Frankl-Tokushige theorem on $k$-wise intersecting families.
\end{abstract}

\maketitle

\tableofcontents
\section{Introduction}

The celebrated Hoffman bound \cite{hoffman2003eigenvalues} connects
spectral graph theory with extremal combinatorics, by upper-bounding
the independence number of a graph in terms of the minimal eigenvalue
of its adjacency matrix. The Hoffman bound, in a generalized version
due to Lovász \cite{lovasz1979shannon}, has seen many applications
in extremal set theory and theoretical computer science.

The Hoffman bound can be used to solve problems in extremal set theory
in which the constraints can be modeled as a graph. As an example,
the Hoffman bound can be used to prove the fundamental Erd\H{o}s\textendash Ko\textendash Rado
theorem on the size of intersecting families, in which the constraint
is that every two sets in the family have nonempty intersection. Other
problems involve more complex constraints, and so are not amenable
to this method. A simple example is the $s$-wise intersecting Erd\H{o}s\textendash Ko\textendash Rado
theorem, due to Frankl \cite{frankl1976sperner}, which concerns families
in which every $s$ sets have nonempty intersection. In this case
the constraints can be modeled as a \emph{hypergraph} rather than as
a graph.

Recently, Hoffman's bound has been generalized to hypergraphs \cite{bachoc2019theta,golubev2016chromatic}. The new bound is particularly attractive for upper-bounding
independent sets in tensor powers of hypergraphs, a setting which
we describe in detail below. We demonstrate the power of this method
by solving a problem of Frankl on triangle-free families and by giving
a spectral proof of Mantel's theorem. We also formulate a number of
known problems in the language of independent sets in hypergraphs.

\subsection{Notations\label{subsec:Notations}}

A multiset is an unordered collection of elements that is allowed to have repetitions, its size is the number of its elements counting the multiplicity. An \emph{$i$-multiset} is a multiset of size $i$. 
Let $V$ be a set. We denote by $V^{[i]}$ the collection of all $i$-multisets of elements of $V$, and elements of $V^{\left[i\right]}$ will be denoted by $\left[v_{1},\ldots,v_{i}\right]$.
$V^{[0]}$ consists of the the empty set.
\begin{defn}
A \emph{weighted $k$-uniform hypergraph} is a pair $X=\left(V,\mu\right)$
where $V$ is the vertex set and $\mu$ is a probability distribution
on $V^{\left[k\right]}.$
\end{defn}
For $0\leq i\leq k-1$, define a probability measure $\mu_{i}$ on
$V^{\left[i\right]}$ by the following process. First, choose a multiset
$\left[v_{1},\ldots,v_{k}\right]$ according to $\mu$, and then choose
an $i$-submultiset of it uniformly at random. We write $X^{\left(i\right)}$
for the set of elements of $V^{\left[i\right]}$ whose $\mu_{i}$
measure is positive. The elements of $X^{\left(i\right)}$ are called
the $i$\emph{-faces} of $X$, and the elements of $X^{(0)}\cup\cdots\cup X^{\left(k\right)}$
are called the \emph{faces} of $X$. Note that if $\sigma_{2}$ is
a face of $X$, then $\sigma_{1}$ is a face of $X$ for any $\sigma_{1}\subseteq\sigma_{2}.$
Note that $X^{(0)}=\{\emptyset\}$, i.e., the empty set is the one
and only $0$-face of $X$. Therefore, the collection of multisets
$\left(X^{\left(k\right)},X^{\left(k-1\right)},\ldots,X^{\left(0\right)}\right)$
can be viewed as an abstract simplicial complex of dimension $(k-1)$
(e.g., as defined in\cite{golubev2016chromatic,bachoc2019theta})
which is however allowed to have loops (multiples of a vertex in a
face). Conversely, an abstract simplicial complex can be made into
a weighted uniform hypergraph by introducing a probability measure
which is positive on its maximal faces.
\begin{defn}
A set $I\subseteq V$ is said to be \emph{independent in a $k$-uniform
hypergraph $X$ }if no $k$-face of $X$ is contained in $I$. The
largest possible value of $\mu_{0}\left(I\right),$ where $I\subseteq V$
is an independent set in $X$, is called the \emph{independence number} of $X$ and denoted $\alpha\left(X\right)$.
A subset $I\subseteq V$ is said to be an \emph{extremal independent
set} of $X$ if $\mu_{0}\left(I\right)=\alpha\left(X\right).$
\end{defn}
We can couple the distributions $\mu_{0},\dots,\mu_{k}$ into a distribution $\boldsymbol{\mu}$ over flags of $X$ by sampling a random $k$-face according to $\mu$ and removing elements from it one by one uniformly at random. It will also be useful to consider distributions $\tilde{\mu}_{i}$ on $V^{[i]}$, obtained by sampling an $i$-face according to $\mu_{i}$, and then choosing a random order of the vertices it contains.

Let $\sigma\in X^{\left(i\right)}$ be an $i$-face. Its \emph{link} in \emph{$X$} is the $(k-i)$-uniform hypergraph $X_{\sigma}=\left(V,\mu_{\sigma}\right),$ where $\mu_{\sigma}$ is the probability distribution that corresponds to the following process: sample a random flag according to $\boldsymbol{\mu}$ subject to $\boldsymbol{\mu}_{i}=\sigma$, and output $\boldsymbol{\mu}_{k}\setminus\sigma$.
For a set $A\in X^{\left(k-i\right)}$ we shall say that $\mu_{\sigma}\left(A\right)$
is the \emph{relative measure of $A$ according to $\sigma.$ }Note
that the link of the empty set is the whole hypergraph $X$ itself.

The \emph{skeleton }of $X$ is the weighted graph $S\left(X\right)$
on the vertex set $X^{\left(1\right)}$, whose edges are $X^{\left(2\right)}$,
and whose weights are given by $\mu_{2}.$ The inner product on the
space $L^{2}\left(X^{(1)},\mu_{1}\right)$ of functions on the vertices
is defined as
\[
\left\langle f,g\right\rangle =\mathbb{E}_{v\sim\mu_{1}}f(v)g(v)=\sum_{v\in X^{(1)}}f(v)g(v)\mu_{1}(v).
\]
The normalized adjacency operator $T_{X}$ of $X$ is that of the
skeleton $S(X)$. In other words, $T_{X}$ acts on $L^{2}\left(X^{(1)},\mu_{1}\right)$
as follows:
\[
\left(T_{X}f\right)(v)=\mathbb{E}_{\mu_{1}(X_{v})}[f].
\]
If $f,g\in L^{2}(X^{(1)},\mu_{1})$ then
\begin{equation}\nonumber
 \begin{split}
 \langle f,T_{X}g\rangle 
 &= \mathbb{E}_{\boldsymbol{v}\sim\mu_{1}}f(\boldsymbol{v})\mathbb{E}_{\boldsymbol{u}\sim\mu_{1}(X_{\boldsymbol{v}})}g(\boldsymbol{u})
 \\ & =\mathbb{E}_{\boldsymbol{w}\sim\boldsymbol{\mu}}f(\boldsymbol{w}_{1})g(\boldsymbol{w}_{2}\setminus\boldsymbol{w}_{1})
 =\mathbb{E}_{\boldsymbol{(u,v)}\sim\tilde{\mu}_{2}}f(\boldsymbol{v})g(\boldsymbol{u}). 
 \end{split}
\end{equation}
This shows that $T_{X}$ is self-adjoint, and so has real eigenvalues.
The matrix form of $T_{X}$ is given by the formula
\begin{equation}\label{eq:norm-adj}
T_{X}(u,v)=\begin{cases}
\frac{\mu_{2}([u,u])}{\mu_{1}(u)} & \text{if }u=v;\\
\frac{\mu_{2}([u,v])}{2\mu_{1}(u)} & \text{if }u\ne v.
\end{cases}
\end{equation}
Similar reasoning shows that we can sample $[u,v]\sim\mu_{2}$ by
sampling $u\sim\mu_{1}$ and $v\sim\mu_{1}(X_{v})$. 

Note that if $V$ is a finite set (which is the case throughout this
paper), then the largest eigenvalue of $T_{X}$ is $1$ and is achieved
on the constant function. By $\lambda(X)$ we denote the smallest
eigenvalue of $T_{X}$. For all $0\leq i<k-2$, we write 
\[
\lambda_{i}\left(X\right)=\min_{\sigma\in X^{\left(i\right)}}\left[\lambda\left(S\left(X_{\sigma}\right)\right)\right].
\]
 In other words, $\lambda_{i}\left(X\right)$ is the minimal possible
value of an eigenvalue of the normalized adjacency matrix of a skeleton
of the link of an $i$-face of $X$. Note that $\lambda_{0}(X)$ is
just the smallest eigenvalue of the normalized adjacency operator
on the skeleton of $X$.
\begin{defn}
The tensor product $X\otimes X'$ of two $k$-uniform hypergraphs
$X=\left(V,\mu\right)$ and $X=\left(V',\mu'\right)$
is a $k$-uniform hypergraph $(V\times V',\mu\times\mu')$,
where $\mu\times\mu'$ stands for the product measure on $(V\times V')^{[k]}\simeq V^{[k]}\times V'^{[k]}$.
For a $k$-uniform hypergraph $X$, we denote by $X^{\otimes n}=\underbrace{X\otimes\dots\otimes X}_{n}$
its $n$-th tensor power.
\end{defn}

\subsection{Results}
We prove a new upper bound for the independence number of a hypergraph,
and its invariance under the tensor power operation.
\begin{thm}
\label{thm:bound-intro} Let $X=\left(V,\mu\right)$ be a $k$-uniform
hypergraph. Then 
\begin{equation}
\alpha\left(X\right)\le1-\frac{1}{\left(1-\lambda_{0}\right)\left(1-\lambda_{1}\right)\cdots\left(1-\lambda_{k-2}\right)}.\label{eq:generalized Hoffman bound-1}
\end{equation}
If in addition $\lambda_{i}\leq0$ for all $0\leq i<k-2$, then for
any positive integer $n$ the following inequality holds for $X^{\otimes n}$:
\[
\alpha(X^{\otimes n})\leq1-\frac{1}{\left(1-\lambda_{0}\right)\left(1-\lambda_{1}\right)\cdots\left(1-\lambda_{d-1}\right)}.
\]
In particular, if the bound is sharp for $X$, it remains sharp for
its tensor powers, as $\alpha(X^{\otimes n})=\alpha(X)$.
\end{thm}

We apply Theorem~\ref{thm:bound-intro} to deduce the following result on  the Frankl's problem on triangle-free families, in both a uniform and a $p$-biased versions.
\begin{thm}
\label{thm:intro-frankl} 
\textbf{The uniform version.} Let $\binom{[n]}{2k}$ be the space of $2k$-subsets
of $[n]$, where $n\le4k-1$. If $\mathcal{F}\subseteq\binom{[n]}{2k}$
is a family of subsets which does not contain three distinct subsets
whose symmetric difference is empty, then $|\mathcal{F}|\le\binom{n-1}{2k-1}$.
This bound is sharp, as, for example, the family of all subsets containing
the element $1$ satisfies the condition and contains $\binom{n-1}{2k-1}$
subsets.

\textbf{The $p$-biased version.} Let $\{0,1\}^{n}$ be the space of $\{0,1\}$-vectors of length $n$
endowed with be the $p$-biased measure $\mu$, where $1/2\leq p\leq2/3$.
If $\mathcal{F}\subseteq\{0,1\}^{n}$ is a family of vectors which
does not contain three distinct vectors whose sum is zero, then $\mu(\mathcal{F})\leq p$.
This bound is sharp, as, for example, the set of all vectors having
$1$ as their first coordinate satisfies the condition and has measure
$p$.
\end{thm}

Our method also provides spectral proves of Mantel's theorem on triangle-free graphs and Frankl-Tokushige theorem on $k$-wise intersecting families.
\begin{thm}
\label{thm:intro-mantel}\cite{mantel1907} If a graph on $n$ vertices
contains no triangle, then it contains at most $\left\lfloor \frac{n^{2}}{4}\right\rfloor$ edges.
\end{thm}

\begin{thm}\cite{frankl2003weighted}
Let $k\ge2$ and $p\leq1-\frac{1}{k}$. Assume $\mathcal{F}\subset\mathcal{P}([n])$ is $k$-wise intersecting family of subsets of $[n]$, that is, for all $F_1,\dots,F_k\in\mathcal{F}$
\[
 F_1\cap\dots\cap F_k \neq \emptyset.
\]
Then $\mu_p(\mathcal{F}) = \sum_{F\in \mathcal{F}} p^{|F|}(1-p)^{n-|F|}\leq p$, in other words, $\mu_p$ is the $p$-biased measure on $\mathcal{P}([n])$.
\end{thm}

\subsection{Structure of the Paper}
In Section~\ref{sec:Hoffman-bound-graphs}, we give a brief overview
of the method for graphs: the Hoffman bound, its behavior for tensor
product of graphs, and applications in extremal combinatorics. In
Section~\ref{sec:Generalization-to-hypergraphs}, we introduce the
required hypergraph definitions and notations, and translate a number
of known problems to the language of independent sets of hypergraphs.
In Section~\ref{sec:High-dimensional-Hoffman}, we prove Theorem \ref{thm:bound-intro} and compare the new bound to the known ones. In Sections~\ref{sec:Frankl-triangle} and~\ref{sec:Mantel-proof}, we prove Theorems~\ref{thm:intro-frankl} and~\ref{thm:intro-mantel}, respectively. \iffalse We conclude the paper with suggestions for further research in Section~\ref{sec:Open-problems}.\fi

\section{Hoffman bound for graphs\label{sec:Hoffman-bound-graphs}}
Let $G=\left(V,\mu\right)$ be a graph, that is, a $2$-uniform hypergraph in the notations of this paper. By edges of $G$ we mean the support of $\mu$ in $V^{[2]}$. A subset of $V$ is called \emph{independent} if it does not contain any edges. Recall that $\mu_1$ is the induced measure on $V$. \iffalse The stationary distribution $\mu_{G}$ on $V$ is the probability distribution in which a vertex is chosen uniformly at random out of a uniformly random edge of from $E$. The measure of a set $A\subseteq V$ is its measure w.r.t. $\mu_{G}$. Note that $\mu_{G}$ coincides with the uniform distribution on $V$ iff $G$ is regular.\fi
The \emph{independence number} of $G$ is the maximal value of $\mu_{1}\left(A\right)$ over all independent sets $A\subseteq V$. If $\mu_{1}\left(A\right)=\alpha\left(G\right)$, we say that $A$ is an \emph{extremal independent set} of $G.$ \iffalse The \emph{normalized adjacency operator} of $G$ is defined by the matrix $P_{G}$ given by $P_{G}\left(v,u\right)=\frac{1}{\deg\left(v\right)}.$ The eigenvalues of $G$ are defined to be the eigenvalues of $P_G$. \fi We write $\lambda_{\min}\left(G\right)$ for the minimal eigenvalue
of the normalized adjacency $T_G$ on $G$, see Equation~\ref{eq:norm-adj} for the formula.

The Hoffman bound for graphs gives an upper bound on the largest independent
set of $G$ in terms of its minimal eigenvalue.
\begin{thm}[Hoffman bound, \cite{hoffman2003eigenvalues}]
 Let $G$ be a graph. Then 
\[
\alpha\left(G\right)\le\frac{-\lambda_{\min}\left(G\right)}{1-\lambda_{\min}\left(G\right)}.
\]
\end{thm}
Ever since Hoffman's original work, the Hoffman bound has become a
central tool in the study of intersection problems in combinatorics.
For example, Lovász,~\cite{lovasz1979shannon} showed that it can be
used to prove the celebrated Erd\H{o}s\textendash Ko\textendash Rado
theorem,~\cite{erdos1961intersection}, (aka EKR theorem), which states
that for $n\geq2k$, the maximum size of an intersecting family of
$k$-element subsets of $[n]$ is $\binom{n-1}{k-1}$. The theorem
follows directly from the Hoffman bound by considering the Kneser
graph, which is the graph on the vertex set $\binom{[n]}{k}$ in which
two sets are connected if they are disjoint. It is crucial for the
proof that the Hoffman bound is tight in this case.

The Hoffman bound can be also used in a more sophisticated manner.
For instance, in~\cite{wilson1984exact}, Wilson considered the $t$-intersection variant of the EKR theorem, in which the goal is to find the largest
subfamily of $\binom{[n]}{k}$ in which any two sets have at least
$t$ elements in common. In contrast to the EKR theorem, in this case
applying Hoffman's bound on the generalized Kneser graph (in which
two sets are connected if their intersection contains less than $t$
points) does not give a tight upper bound. Instead, one needs to accurately choose weights for the edges of this graph, some of them negative,
and only then apply the Hoffman bound. In this way, Wilson managed
to determine all values of $n,k,t$ in which the extremal family is
the family of all sets that contain a given set of size $t$, namely,
$n\ge\left(t+1\right)\left(k-t+1\right)$.

This approach became more systematic in the work of Friedgut,~\cite{friedgut1998boolean},
who applied Fourier analysis to construct a matrix for the $p$-biased
version of Wilson theorem. This Fourier-analytic approach was also
useful in the proof of a long-standing problem of Simonovitz and Sós
on triangle-intersecting families of graphs,~\cite{ellis2012triangle}.
Finally, Ellis, Friedgut, and Pilpel,~\cite{ellis2011intersecting},
used a similar approach to solve an old problem of Deza and Frankl,
showing that a $t$-intersecting family of permutations in $S_{n}$
contains at most $(n-t)!$ permutations, for large enough $n$ (two
permutations $t$-intersect if they agree on the image of at least
$t$ points).

The main flaw of the Hoffman bound is the fact that it does not apply
to problems in which the constraint involves more than two sets. For
example, it cannot be used to upper-bound the size of a subfamily
of $\binom{[n]}{k}$ in which any \emph{three} sets $t$-intersect.
It is therefore of great importance to find high-dimensional analogues
of the Hoffman bound that fit these more general restrictions.

Over the years there has been great interest in the problem of generalizing results from graphs to hypergraphs (or, equivalently, simplical complexes), see e.g. \cite{AnnaGundert2015,parzanchevski2017simplicial,parzanchevski2017mixing,golubev2019spectrum,lubotzky2014ramanujan,linial2016phase,parzanchevski2016isoperimetric}.
In particular, two generalizations of the Hoffman bound were given
by \cite{bachoc2019theta} and \cite{golubev2016chromatic}. Such
results are known as \emph{high-dimensional Hoffman bounds}. The goal
of this paper is to give a new high-dimensional Hoffman bound, which
we believe is the right tool for tackling many problems in extremal
combinatorics.

\subsection{Independent sets in tensor power of graphs}

The Hoffman bound is particularly useful for independent sets in \emph{tensor powers of graphs}. Given graphs $G=\left(V,\mu\right)$, $G'=\left(V',\mu'\right)$, their tensor product is the graph $G\otimes G'$ whose vertex set is $V\times V'$ endowed with the product measure $\mu\times\mu'$. In particular, two of vertices in the product are connected
by an edge if they are connected by an edge in each coordinate. \iffalse More
generally, $G_{1},G_{2}$ are weighted graphs with loops, and $G_{1}\otimes G_{2}$
is the graph where the weight of the edge $\left\{ \left(v_{1},u_{1}\right),\left(v_{2},u_{2}\right)\right\} $
is the product of the weights of the corresponding edges of $G_{1},G_{2}.$ \fi
The $n$-th tensor power of $G$ is the graph $G^{\otimes n}=G\otimes\cdots\otimes G$.
Independent sets in tensor products of graphs are well studied, see
e.g. \cite{alon2004graph}, \cite{alon2007independent}, and \cite{dinur2008independent}.\iffalse
There are several motivations for studying independent sets in tensor
powers. One motivation for these problems is their connection to the
Hedetniemi conjecture \cite{hedetniemi1966homomorphisms}, which states
that $\chi\left(G_{1}\times G_{2}\right)=\min\left\{ \chi\left(G_{1}\right),\chi\left(G_{2}\right)\right\} $,
see Alon and Lubetzky \cite{alon2007independent} for more details
(TODO disproved by Yaroslav Shitov).
\fi 
One motivation is their connection to the theory of hardness of approximation, see~\cite{dinur2009conditional}. However, our main focus in this paper will be the connection to extremal set theory. This connection was first implicitly established by Friedgut \cite{friedgut2008measure}, and later more explicitly by Dinur and Friedgut \cite{dinur2009intersecting}. 

It is a well-known rule of thumb (backed by various results) that
the two Erd\H{o}s\textendash Rényi models of random graphs, $G\left(n,p\right)$
and $G\left(n,m\right)$, should behave similarly when $p=\frac{m}{n}$.
A similar phenomenon holds in extremal set theory: questions about
subfamilies of $\binom{[n]}{k}$ behave similarly to questions about
subsets of $\mathcal{\mathcal{P}}([n])$ with respect to the $p$-biased
measure $\mu_{p}$ for $p=k/n$, where a set $\boldsymbol{A}\subseteq\left[n\right]$
is chosen by independently putting each element of $\left[n\right]$
inside $\boldsymbol{A}$ with probability $p$ and outside of it with
probability $1-p$. We write $\mu_{p}\left(\mathcal{F}\right)$ for
the probability that a random subset $\boldsymbol{A}\sim\mu_{p}$
belongs to $\mathcal{F}.$

For instance, the EKR theorem, which asks how large can an intersecting
family $\mathcal{F}\subseteq\binom{\left[n\right]}{k}$ be, can be
rephrased as follows: 
\begin{quotation}
Suppose that $\mathcal{F}$ is intersecting.
How large can the probability that a random $k$-set belongs to $\mathcal{F}$
be?
 \end{quotation}
This formulation of the problem immediately suggests the following
$p$-biased analogue:
\begin{quotation}
How large can $\mu_{p}\left(\mathcal{F}\right)$
be if $\mathcal{F}\subseteq\mathcal{P}\left(\left[n\right]\right)$
is intersecting? 
\end{quotation}
This problem was first studied by Ahlswede and Katona,~\cite{ahlswede1977contributions}, in the 70's. While both the EKR
problem and its $p$-biased analogue have already been solved, this
connection is still useful for many of the generalizations of the
EKR theorem, where a result in the $k$-uniform setting can be deduced
from the $p$-biased setting and vice versa, see \cite{dinur2005hardness}.

The $p$-biased version of the EKR theorem is the problem of determining
the extremal independent sets in the graph $G^{\otimes n},$ where
$G$ consists of two vertices $\left\{ 0,1\right\} $ with the (undirected)
edge $\left\{ 1,0\right\} $ having the weight $2p,$ and with the
edge $\left\{ 0,0\right\} $ having the weight $1-2p.$ This observation
was used by Friedgut,~\cite{friedgut1998boolean}, Dinur and Friedgut,~\cite{dinur2005hardness}, and later by Friedgut and Regev,~\cite{friedgut2017kneser}, to apply Fourier-analytic methods that are natural in the context of graph products in order to study variants of the EKR theorem.

In a suitable basis, the matrix of the adjacency operator of a tensor product of graphs is the Kronecker product of the adjacency operator matrices of the factors. Hence, the following property holds
\[
\lambda_{\min}\left(G^{\otimes n}\right)=\begin{cases}
\lambda_{\min}\left(G\right) & \text{if }\lambda_{\min}\left(G\right)\le0;\\
\lambda_{\min}\left(G\right)^{n} & \text{if }\lambda_{\min}\left(G\right)\ge0.
\end{cases}
\]
This immediately implies that the Hoffman bound is sharp on $G^{\otimes n}$, whenever it is sharp on $G$ (the Hoffman bound is never sharp if
$\lambda_{\min}$ is positive). This reduces the $p$-biased EKR problem
for families $\mathcal{F}\subseteq\mathcal{P}\left(\left[n\right]\right)$
to the problem of showing that the Hoffman bound is sharp in the special
case $n=1$, which can be verified directly.

A subset $A$ of $G^{\otimes n}$ is called a \emph{dictatorship}
if there exists a set $B\subseteq G$ and $1\leq i\leq n$ such that a vertex $x=(x_1,\dots,x_n)$ is in $A$ iff $x_{i}$ is in $B$. The above observation shows that
if the Hoffman bound is sharp for $G$, then the Hoffman bound is
sharp for $G^{\otimes n}$ as well, and the dictatorships corresponding
to extremal independent sets of $G$ are extremal for $G^{\otimes n}$
(not necessarily exclusively). Alon, Dinur, Friedgut, and Sudakov,~\cite{alon2004graph}, showed the following stronger version of this
observation.
\begin{thm}[\cite{alon2004graph}]
 Let $G$ be a weighted connected non-bipartite graph. If the Hoffman
bound is sharp for $G$, i.e., $\alpha\left(G\right)=\frac{-\lambda_{\min}}{1-\lambda_{\min}}$,
then $\alpha\left(G^{\otimes n}\right)=\alpha\left(G\right).$ Moreover,
if $A$ is an independent set with $\mu_{G}\left(A\right)=\alpha\left(G\right)$,
then $A$ is a dictatorship.

Suppose additionally that $\mu_{G}\left(\left\{ v\right\} \right)=\Theta\left(1\right)$
for each $v\in G.$ Then for each $\epsilon>0$ there exists $\delta>0$
such that if an independent set $A$ satisfies $\mu_{G}\left(A\right)>\alpha\left(G\right)-\delta,$
then there exists an independent dictatorship $B$ such that $\mu_{G}\left(A\Delta B\right)<\epsilon.$
\end{thm}

\section{Known problems in the language of hypergraphs\label{sec:Generalization-to-hypergraphs}}

There are plenty of reasons why the above theory begs to be generalized
to hypergraphs (or, equivalently, simplical complexes). In addition to the definitions given in the introduction, let us give a version of them in the $k$-partite setting.
\begin{defn}
A weighted $k$-partite hypergraph is a tuple $X=\left(V_{1},\ldots,V_{k},\mu\right)$,
where $\mu$ is a probability distribution on $V_{1}\times\cdots\times V_{k}$.
The probability distribution $\mu_{X,V_{i}}$ is the probability
distribution on $V_{i}$, where a vertex is chosen as the projection
on $V_{i}$ of a random element chosen according to $\mu.$ Sets $A_{1}\subseteq V_{1},\ldots,A_{k}\subseteq V_{k}$
are said to be \emph{cross-independent} if the probability that a
random element of $\mu$ belongs to $\prod_{i=1}^{k}A_{i}$ is $0$.
The tensor power of $X$ is the $k$-partite hypergraph
$X^{\otimes n}=\left(V_{1}^{n},\ldots,V_{k}^{n},\mu^{\otimes n}\right),$
where $\mu^{\otimes n}$ is the product probability distribution.
\end{defn}

Independent sets in tensor powers of hypergraphs arise all over combinatorics,
and many of the fundamental problems in extremal combinatorics can
be formulated as problems about independent sets in tensor powers
of hypergraphs. Below we formulate several well-known problems, solved
and open, in the language of hypergraphs, and give proof to two of
them: Frankl's Triangle Problem and Mantel's Theorem.

\subsection{Number theory}

How large can a subset $A\subseteq\mathbb{F}_{q}^{n}$ be if it does
not contain elements $x_{1},\ldots,x_{m}\in A$ that form a solution
to the system of equations $\sum_{i=1}^{m}a_{ij}x_{i}=b_{j}$, for
parameters $a_{ij},b_{j}\in\mathbb{F}_{q}$? For instance, the Meshulam\textendash Roth
theorem,~\cite{meshulam1995subsets}, concerns the problem of determining
how large can a subset of $\mathbb{F}_{p}^{n}$ be if it contains
no non-trivial solutions to the equation $x+z=2y.$ Similarly, the
analogous problem for $k$-term arithmetic progressions can be formulated
as a non-trivial solution to the equations 
\[
x_{1}+x_{3}=2x_{2},x_{2}+x_{4}=2x_{3},\ldots,x_{m-2}+x_{m}=2x_{m-1}.
\]

Take the hypergraph $X=\left(\mathbb{F}_{q},\mu\right),$
where $\mu$ is positive on the solutions to the equations $\sum_{i=1}^{m}a_{ij}x_{i}=b_{j}.$
Since a solution to the system of equations in $\mathbb{F}_{q}^{n}$
is a solution iff it is a solution in each coordinate,
the hypergraph $X^{\otimes n}$ is the hypergraph whose
independent sets correspond to solutions of the same system of equations as above.
This observation was first given by Mossel,~\cite{mossel2010noise}.

\subsection{Turán problem for hypergraphs}

One of the fundamental problems in extremal combinatorics is the Turán
problem for hypergraphs, which asks how large can a family $\mathcal{F}\subseteq\binom{\left[n\right]}{k}$
be if it does not contain a copy of some given hypergraph $H.$ The
Turán problem can be restated as a problem about independent sets
in tensor powers of $k$-partite-hypergraphs, see~\cite{lifshitz2018nearly}.
Let us state it in the case of triangles in graphs.
\begin{example}
\label{exa:Mantel hypergraph} Let $K_{2,2,2}$ be the complete $3$-partite
hypergraph between sets $A_{1},A_{2},A_{3}$ of size 2, let $E_{ij}$
be the set of edges between $A_{i}$ and $A_{j}$, and let $X=\left(V_{1},V_{2},V_{3},\mu\right)$
be the $3$-partite hypergraph, where 
\[
V_{1}=E_{23},V_{2}=E_{31},V_{3}=E_{12},
\]
 and where $\mu$ is the uniform measure on the set of triangles in $K_{2,2,2}.$ The hypergraph
\[
X^{\otimes n}
\]
is the hypergraph whose vertices correspond to the edges in $K_{2^{n},2^{n},2^{n}}$
and whose edges correspond to the triangles in $K_{2^{n},2^{n},2^{n}}.$
Therefore, cross-independent sets correspond to three directed graphs
\[
G_{1}\subseteq V_{1}\times V_{2},G_{2}\subseteq V_{2}\times V_{3},G_{3}\subseteq V_{3}\times V_{1},
\]
 where each $V_{i}$ is $\left\{ 0,1\right\} ^{n}.$
\end{example}
As mentioned above, this example can be generalized to arbitrary hypergraphs.
In Section~\ref{sec:Mantel-proof}, we shall use this construction
to give a spectral proof of Mantel's theorem for graphs with $2^{n}$
vertices. In this context, Mantel's theorem can be restated as follows:
\begin{thm}
Let $G_{1},G_{2},G_{3}$ be cross-independent sets in $X^{\otimes n}.$
Suppose additionally that $G_{1},G_{2},G_{3}$ are all equal to some
bipartite graph $G\subseteq\left(\left\{ 0,1\right\} \times\left\{ 0,1\right\} \right)^{n}$,
and that $G$ corresponds to a graph in the sense that $\left(a,b\right)\in G$
if and only if $\left(b,a\right)\in G$ (in other words, $G$ is the
bipartite cover of some graph). Then the largest value of $\left|G\right|$
is attained when $G$ is the dictatorship of all $x\in\left(\left\{ 0,1\right\} \times\left\{ 0,1\right\} \right)^{n}$
whose first coordinate is either $\left(0,1\right)$ or $\left(1,0\right)$.
\end{thm}

Indeed, the dictatorship of all $x\in\left(\left\{ 0,1\right\} \times\left\{ 0,1\right\} \right)^{n}$
whose first coordinate is either $\left(0,1\right)$ or $\left(1,0\right)$
corresponds to a balanced complete bipartite graph, which is extremal
for Mantel's theorem. Interestingly enough, the dictatorships contain
only some of the complete balanced bipartite graphs, but not all of
them.

\subsection{Extremal set theory}

Many $p$-biased versions of problems in extremal set theory can be
described as a special case of the problem: ``How large can $\alpha\left(X^{\otimes n}\right)$
be?'', where $X$ is a weighted hypergraph. 

\subsubsection{Erd\H{o}s Matching Conjecture}
Our first example
is the Erd\H{o}s Matching Conjecture,~\cite{erdHos1965problem}, from
1965. An\emph{ $s$-matching} is a family of $s$ sets $\left\{ A_{1},\ldots,A_{s}\right\} $
that are pairwise disjoint. Erd\H{o}s asked how large can a family
$\mathcal{F}\subseteq\binom{\left[n\right]}{k}$ be if it does not
contain an $s$-matching. He conjectured that the extremal family
is either the family of all sets that intersect a given set of size
$s-1$ in at least one element, or the family $\binom{\left[ks-1\right]}{k}.$
The corresponding $p$-biased version of this problem is as follows:
\begin{quotation}
Given $p\le\frac{1}{s}$, how large can $\mu_{p}\left(\mathcal{F}\right)$
be if $\mathcal{F}$ does not contain an $s$-matching?
\end{quotation}
This is the problem of determining the independence number of the
$n$-th tensor power of the $s$-uniform hypergraph whose vertex set
is $\left\{ 0,1\right\} $ with the weight function $\mu([1,0,\dots,0])=sp$
and $\mu\left([0,\ldots,0]\right)=1-sp$ (recall that $[\cdot]$ is
our notation for multiset). A nice feature of the $p$-biased variant
of the Erd\H{o}s Matching Conjecture is that there is only one suggestion
for the extremal family, which is the family of all sets that intersect
a given set of size $s-1$ in at least one element.

\subsubsection{$s$-wise Intersecting Families}

The second example is the problem of $s$-wise intersecting families,
first studied by Frankl,~\cite{frankl1976sperner}. A family $\mathcal{F}\subseteq\binom{\left[n\right]}{k}$
is \emph{$s$-wise intersecting }if the intersection of every $s$
sets in $\mathcal{F}$ is nonempty. Frankl showed that when $k\le\tfrac{s-1}{s}n$,
the extremal $s$-wise intersecting family is the family of all sets
that contain a given element (otherwise every family is $s$-wise
intersecting). The $p$-biased version of the problem was studied
by Frankl and Tokushige,~\cite{frankl2003weighted}. They showed that
the largest value of $\mu_{p}\left(\mathcal{F}\right)$ for an $s$-wise
intersecting family $\mathcal{F}$ is $p$, as long as $p\leq\frac{s-1}{s}$.
This problem can be expressed as the determining the independence
number of the hypergraph $X^{\otimes n}$, where
\begin{enumerate}
\item The hypergraph $X=\left(V,\mu\right)$ has $\left\{ 0,1\right\} $
as its vertex set $V$.
\item The induced distribution $\mu_{1}$ on $V$ is the $p$-biased one.
\item $\mu\left(x\right)=0$ if and only $x$ is the all ones vector.
\end{enumerate}
It is easy to construct many hypergraphs $X$ that satisfy
these hypotheses. We reprove this result in \ref{subsec:Sharpness-of-tensor}.

\subsubsection{Frankl's Turán Problem}

Last but not least, this problem is related to Frankl's Turán problem
on hypergraphs without extended triangles. A \emph{triangle} in $\mathcal{P}\left(\left[n\right]\right)$
is a $2k$-uniform hypergraph $\left\{ A,B,C\right\} $ such that
each element of $\left[n\right]$ belongs to an even number of the
sets $A,B,C.$ In other words, there exist disjoint $k$-element sets
$D,E,F$ such that $D\cup E=A,$ $D\cup F=B,$ and $E\cup F=C.$ Frankl,~\cite{frankl1990asymptotic}, asked how large can a family $\mathcal{F}\subseteq\binom{\left[n\right]}{2k}$
be if it does not contain a triangle. The reason for considering only
even uniformities is that no $k$-uniform triangle exists for an odd
$k$. The $p$-biased version of the problem is as follows:
\begin{quotation}
Given $p\le\frac{2}{3}$, how large can $\mu_{p}\left(\mathcal{F}\right)$
be if $\mathcal{F}\subseteq\mathcal{P}\left(\left[n\right]\right)$
does not contain a triangle?
\end{quotation}
The reason for the condition $p\le\frac{2}{3}$ is the fact that
the family $\left\{ A:\,\left|A\right|>\frac{2}{3}n\right\} $ is
triangle-free, and its $p$-biased measure tends to $1$ as $n$ tends
to infinity.

The $p$-biased version of Frankl's Turán problem is the problem of
determining the independence number of the hypergraph $X^{\otimes n},$
where $X=\left(V,\mu\right)$ is with $V=\left\{ 0,1\right\} $
and
\[
\mu\left([1,1,0]\right)=\frac{3}{2}p,\:\mu\left([0,0,0]\right)=1-\frac{3}{2}p.
\]
In Section \ref{sec:Frankl-triangle} we prove the following theorem:
\begin{thm}
If a family of $2k$-subsets of $[n]$ contains no three distinct
subsets whose symmetric difference is empty, then the family contains
at most $\frac{2k}{\min(n,4k-1)}\binom{n}{2k}$ subsets. 

Furthermore, when $n\le4k-1$ and $p\ge1/2$ the bounds
are tight for ``dictatorships'' (all subsets or vectors containing
a specific point), and otherwise the bounds are asymptotically tight,
in the $p$-biased case for the family of all vectors having odd parity,
and in the $2k$-uniform case for the family of all subsets whose
intersection with $[\lfloor n/2\rfloor]$ is odd.

Similarly we prove that if a subset of $\{0,1\}^{n}$ contains no three distinct vectors summing to zero, then for all $p\le2/3$, its $\mu_{p}$-measure is at most $\max(p,1/2)$.
\end{thm}

\section{High dimensional Hoffman bound\label{sec:High-dimensional-Hoffman}}

Suppose we are given a problem in extremal set theory where constraints
on more than two elements are involved. A possible strategy for solving
it is to first incorporate families $\mathcal{F}$ that satisfy the
constraint as independent sets in some hypergraph or simplical complex,
and then to find and apply a high-dimensional generalization of the
Hoffman bound in order to bound the size of $\mathcal{F}$. Two such
generalizations of the Hoffman bound were obtained recently by Golubev,~\cite{golubev2016chromatic}, and by Bachoc, Gundert, and Passuello,~\cite{bachoc2019theta}. However, none of them seem to give sharp
results in our problems of interest. Instead, we develop a different
generalization of the Hoffman bound in the spirit of \cite{golubev2016chromatic}.

Let $X=(V,\mu)$ be a $k$-uniform hypergraph on the vertex set $V$.
Recall (see Subsection \ref{subsec:Notations}) that for $0\leq i\leq k-2$
we denote by $\lambda_{i}\left(X\right)$ the minimal possible value
of an eigenvalue of the normalized adjacency matrix of the skeleton
of the link of an $i$-face of $X$. That is,
\[
\lambda_{i}\left(X\right)=\min_{\sigma\in X^{\left(i\right)}}\left[\lambda\left(S\left(X_{\sigma}\right)\right)\right].
\]
Note that the hypergraph itself is the link of the only $0$-face,
the empty set, and hence $\lambda_{0}(X)$ is just the smallest eigenvalue
the normalized adjacency operator on the skeleton of $X$.

\subsubsection*{Example}

Let $X=\left(\left\{ 1,2\right\} ,\mu\right)$ be a graph (in other
words, $2$-uniform hypergraph) on two vertices $\left\{ 1,2\right\} $
with the probability measure $\mu$ defined on the edges as 
\[
\mu\left(\left[1,1\right]\right)=p_{1},\:\mu\left(\left[1,2\right]\right)=p_{2},\:\mu\left(\left[2,2\right]\right)=p_{3}.
\]
Then the induced distribution $\mu_{1}$ on the vertices is as follows:
\[
\mu_{1}(1)=p_{1}+\frac{1}{2}p_{2},\:\mu_{1}(2)=\frac{1}{2}p_{2}+p_{3}.
\]
The normalized adjacency operator $T_{X}$ on the skeleton of $X$
has the matrix form
\[
T_{X}=\left(\begin{array}{cc}
\frac{p_{1}}{p_{1}+\frac{1}{2}p_{2}} & \frac{\frac{1}{2}p_{2}}{p_{1}+\frac{1}{2}p_{2}}\\
\frac{\frac{1}{2}p_{2}}{\frac{1}{2}p_{2}+p_{3}} & \frac{p_{3}}{\frac{1}{2}p_{2}+p_{3}}
\end{array}\right)
\]
and while its largest eigenvalue is equal to $1$, the smallest one
is equal to
\[
\lambda_{0}(X)=1-\frac{2p_{2}}{1-(p_{1}-p_{3})^{2}}.
\]

\begin{thm}
\label{thm:Generalized-Hoffman-bound} Let $X=\left(V,\mu\right)$
be a $k$-uniform hypergraph. Then 
\begin{equation}
\alpha\left(X\right)\le1-\frac{1}{\left(1-\lambda_{0}\right)\left(1-\lambda_{1}\right)\cdots\left(1-\lambda_{k-2}\right)}.\label{eq:generalized Hoffman bound}
\end{equation}
\end{thm}
\iffalse
We shall also discuss the problem of determining the extremal independent
in the case where equality hold in (\ref{eq:generalized Hoffman bound}).\fi
\begin{proof}
The proof goes by induction on the uniformity of the hypergraph. The
base case, $k=2$, is the graph case, and the bound \ref{eq:generalized Hoffman bound}
reads as the classical Hoffman bound. Assume that the bound holds
for $(k-1)$-uniform hypergraphs. Let $T_{X}$ be the normalized adjacency
operator of the skeleton of $X$, and let $v_{1}=1,v_{2},\ldots,v_{m}$
an orthonormal basis of its eigenvectors with eigenvalues $1\ge l_{2}\ge\cdots\ge l_{m}=\lambda_{0}$
(recall that $T_{X}$ is self-adjoint). Let $I$ be an independent
set of measure $\alpha(X)$, and let $f=1_{I}$ be its indicator function.
We may write 
\[
f=\sum_{i=1}^{m}\left\langle f,v_{i}\right\rangle v_{i}.
\]
 On the one hand, 
\[
\left\langle T_{X}f,f\right\rangle =\text{Pr}_{[\boldsymbol{x},\boldsymbol{y}]\sim\mu_{2}}\left[\boldsymbol{x},\boldsymbol{y}\in I\right],
\]
or in other words, it is equal to the probability of an ordered edge
($1$-face) distributed according to $\mu_{2}$ to have both ends
in $I$. On the other hand, 
\begin{align*}
\left\langle T_{X}f,f\right\rangle  & =\sum_{i=1}^{m}l_{i}\left\langle f,v_{i}\right\rangle ^{2}\\
 & \ge\left\langle f,1\right\rangle ^{2}+\sum_{i=2}^{m}\lambda_{0}\left\langle f,v_{i}\right\rangle ^{2}\\
 & =\left\langle f,1\right\rangle ^{2}\left(1-\lambda_{0}\right)+\lambda_{0}\left\langle f,f\right\rangle \\
 & \mathbb{=E}\left[f\right]^{2}\left(1-\lambda_{0}\right)+\lambda_{0}\mathbb{E}\left[f^{2}\right].
\end{align*}
Since $f$ is an indicator function, $\mathbb{E}[f]=\mathbb{E}[f^{2}]=\alpha(X)$,
and so
\[
\text{Pr}_{[\boldsymbol{x},\boldsymbol{y}]\sim\mu_{2}}\left[\boldsymbol{x},\boldsymbol{y}\in I\right]\geq\alpha(X)^{2}\left(1-\lambda_{0}\right)+\lambda_{0}\alpha(X).
\]
Note that 
\[
\text{Pr}_{[\boldsymbol{x},\boldsymbol{y}]\sim\mu_{2}}\left[\boldsymbol{x},\boldsymbol{y}\in I\right]\leq\alpha(X)\cdot\max_{x\in I}\text{Pr}_{\boldsymbol{y}\sim\mu_{1}(X_{x})}\left[\boldsymbol{y}\in I\right]
\]
and that for a fixed vertex $x$, the probability $\text{Pr}_{y\sim\mu_{1}(X_{x})}\left[\boldsymbol{y}\in I\right]$
is the measure of an independent set of vertices in its link $X_{x}$,
which is a $(k-1)$-uniform hypergraph. By the induction assumption,
\[
\text{Pr}_{y\sim\mu_{1}(X_{x})}\left[\boldsymbol{y}\in I\right]\leq1-\frac{1}{\left(1-\lambda_{1}\right)\cdots\left(1-\lambda_{k-2}\right)}.
\]
Combining the above bounds, we deduce that

\[
(1-\lambda_{0})\alpha(X)+\lambda_{0}\leq\text{Pr}_{y\sim\mu_{1}(X_{x})}\left[\boldsymbol{y}\in I\right]\leq1-\frac{1}{\left(1-\lambda_{1}\right)\cdots\left(1-\lambda_{k-2}\right)},
\]
which proves the required bound after rearrangement.
\end{proof}
\iffalse
\begin{example}
If $X$ is the complete $k$-uniform hypergraph on $n$ vertices (equivalently,
the complete $(k-1)$-dimensional simplicial complex) equipped with
the uniform measure on ${[n] \choose k}$, then $X^{\left(i\right)}$
consists of the element $\left[v_{1},\ldots,v_{i}\right]$, where
the $v_{j}$ are distinct, and the graph $S\left(X_{\sigma}\right)$
is the complete graphs on the vertices $\left[n\right]\backslash\left\{ v_{1},\ldots,v_{i}\right\} $
equipped with the uniform distribution. Since all skeletons of links
are complete graphs, we easily compute $\lambda_{i}=-\frac{1}{n-1-i}$,
and so the generalized Hoffman bound gives the sharp bound

\[
\alpha(X)\leq1-\frac{1}{\frac{n}{n-1}\cdot\frac{n-1}{n-2}\cdot\cdots\cdot\frac{n-k+2}{n-k+1}}=1-\frac{n-k+1}{n}=\frac{k-1}{n}.
\]
\end{example}
\fi

\subsection{Comparison to prior work}

There are two known spectral upper bounds on the independence number
of a hypergraph (equivalently, a simplicial complex): one in \cite{golubev2016chromatic},
to which we will refer as \emph{the Laplacian bound}, and one in \cite{bachoc2019theta},
to which we will refer as \emph{the Theta bound}. To the bound in
Theorem \ref{thm:Generalized-Hoffman-bound} we will refer as \emph{the
Link bound}.

The Link bound and the Laplacian bound are based on the same idea.
Namely, the bound on the size of an independent set is obtained via
a combination of a lower and an upper bound on the number of $2$-faces
between the maximal independent set and its complement. Below we provide
an explanation why the Link bound is always as least as good as the
Laplacian bound.

The Theta bound follows a different approach. In \cite{bachoc2019theta},
the authors show that the Theta bound is incomparable to the Laplacian
bound by providing two families of simplicial complexes, on which one
bound is sharp while the other is not, and vice versa. However, the
new bound is sharp for all examples provided in \cite{bachoc2019theta}.
In other words, the Link bound is better than the Theta bound in at least one case, and they are equal in other cases. It is not clear whether these bounds are comparable or not.

In order to compare the Laplacian bound with the Link one, let us reformulate the former in the terms of the normalized Laplacian (see, e.g.,~\cite{horak2013spectra} for the definition of the normalized Laplacian). Let $X$ be a $k$-uniform
hypergraph. Then
\[
\alpha\left(X\right)\le1-\frac{1}{\mu_{0}\dots\mu_{k-2}},
\]
where $\mu_{i}$ is the largest eigenvalue of the normalized $i$-Laplacian
on $X$. In order to show that the Link bound is at least as good as
the Laplacian one, it suffices to show that $(1-\lambda_{i})\leq\mu_{i}$
for all $i=0,\dots,k-2$. For graphs, the normalized Laplace operator
$\Delta$ and the normalized adjacency operator $T$ satisfy $\Delta=Id-T$,
and hence $\lambda_{0}=1-\mu_{0}$. For $0 < i$, the Laplace
bound exploits the normalized Laplace operators on the whole hypergraph,
while the Link bound takes the minimum over the smallest eigenvalues
of the adjacency operators on the links. For $2\leq i\leq k-2$,
there exists a function $f_{i}$ supported on the vertices of the
link of an $(i-1)$-cell of the complex $X$ and of norm $1$ which
is an eigenfunction of the normalized adjacency operator on the link
with eigenvalue $\lambda_{i}$. Given such $f_{i}$, one can define
an $i$-cochain $\widehat{f_{i}}$ on $X$ supported in the star of
an $(i-1)$-cell, such that $\langle\Delta\widehat{f_{i}},\widehat{f_{i}}\rangle=1-\lambda_{i}$. Since $\mu_{i}$ is the largest eigenvalue of
the Laplacian and $\widehat{f_{i}}$ is of norm $1$, $\mu_{i}\geq1-\lambda_{i}$.

\subsection{Sharpness of Hoffman bound on $X$ implies sharpness in $X^{\otimes n}$\label{subsec:Sharpness-of-tensor}}

The goal of this section is to prove that the bound~\ref{eq:generalized Hoffman bound} remains sharp for tensor powers of a hypergraph if it is sharp for the hypergraph itself given all the minimal eigenvalues are negative.
First recall the following definition.
\begin{defn}
The tensor product $X\otimes X'$ of two $k$-uniform hypergraphs
$X=\left(V,\mu\right)$ and $X'=\left(V',\mu'\right)$
is a $k$-uniform hypergraph $(V\times V',\mu\times\mu')$,
where $\mu\times\mu'$ stands for the product measure on $(V\times V')^{[k]}\simeq V^{[k]}\times V'^{[k]}$.
For a $k$-uniform hypergraph $X$, we denote by $X^{\otimes n}=\underbrace{X\otimes\dots\otimes X}_{n}$
its $n$-th tensor power. 
\end{defn}

\begin{prop}
\label{prop:ten-prod-evs}Let $X=X\otimes X'$ be the tensor
product of $k$-uniform hypergraphs $X$ and $X'$. Then for
all $0\leq i<k-2$, the following holds for the smallest eigenvalues
of the normalized adjacency operator on the links of its $i$-faces:
\[
\lambda_{i}(X)
=\min_{\sigma\in X^{\left(i-1\right)}}\left[\lambda\left(S\left(X_{\sigma}\right)\right)\right]
=\begin{cases}
\lambda_{i}\left(X\right)\lambda_{i}\left(X'\right), & \text{if }\lambda_{i}\left(X_{j}\right)\geq0\text{ for }j=1,2;\\
\min\left\{ \lambda_{i}\left(X\right),\lambda_{i}\left(X'\right)\right\} , & \text{otherwise}.
\end{cases}
\]
In particular, for the tensor power it reads as
\[
\lambda_{i}(X^{\otimes n})
=\begin{cases}
\lambda_{i}\left(X\right), & \text{if }\lambda_{i}\left(X\right)\le0;\\
\lambda_{i}\left(X\right)^{n}, & \text{if }\lambda_{i}\left(X\right)\ge0.
\end{cases}
\]
\end{prop}

\begin{proof}
The proposition is a combination of the following two facts. The first
is that for an $i$-face $\sigma=(\sigma_{1},\sigma_{2})\in X^{(i)}$,
its link is the tensor product of the links, i.e.,$X_{\sigma}=X_{\sigma_{1}}\otimes X_{\sigma_{2}}$.
The second one is that the matrix of the normalized adjacency operator
of the tensor product of two graphs is the Kronecker product of the
corresponding matrices of the factors. The result, which dates back
to Kronecker himself, states that the eigenvalues of the Kronecker
product are exactly the products of the eigenvalues of the factors (see~\cite{weichsel1962kronecker} for the proof in the uniform case). Finally, since $T_{X}$ is a Markov matrix, all of its
eigenvalues are bounded in magnitude by $1$, which implies the stated
formula.
\end{proof}
\begin{thm}
Let $X=(V,\mu)$ be a $k$-uniform hypergraph such that $\lambda_{i}\leq0$
for all $0\leq i<k-2$ and such that the bound (\ref{eq:generalized Hoffman bound})
is sharp for it, i.e.,
\[
\alpha(X)=1-\frac{1}{\left(1-\lambda_{0}\right)\left(1-\lambda_{1}\right)\cdots\left(1-\lambda_{k-2}\right)},
\]
Then it is also sharp for $X^{\otimes n}$ for any positive integer
$n$, and 
\[
\alpha(X^{\otimes n})=1-\frac{1}{\left(1-\lambda_{0}\right)\left(1-\lambda_{1}\right)\cdots\left(1-\lambda_{k-2}\right)}.
\]
\end{thm}

\begin{proof}
It is a direct corollary of Proposition \ref{prop:ten-prod-evs} that
the r.h.s. of the bound (\ref{eq:generalized Hoffman bound}) is the
same for $X^{\otimes n}$ as for $X$. In order to show that it is
sharp, note that if $I\subseteq V$ is a maximal independent set in
$X$, that is $\mu(I)=\alpha(X)$, then the set $(I,V,\dots,V)\subseteq V^{n}$
is independent in $X^{\otimes n}$.
\end{proof}

\subsection{Computing the generalized Hoffman bound}

Given a $k$-uniform hypergraph $X$ on the vertex set $V$
and a distribution $\nu$ on $V$, the generalized Hoffman bound gives
an upper bound on the $\nu$-measure of an independent set of $X$
for each $k$-uniform weighted hypergraph $X=(V,\mu)$ whose weight
function satisfies the following two constraints: $\mu_{1}=\nu$ and
$\mu(x)=0$ whenever $x\notin X$. We can formulate the best
bound obtainable in this way as a problem whose variables
are the entries of $\mu$:
\[
\begin{aligned}\min\: & (1-\lambda_{0})\cdots(1-\lambda_{k-2})\\
s.t.\: & T_{X_{s}}\succeq\lambda_{|s|}\mathrm{Id}\\
 & \mu_{1}=\nu\\
 & \mu(x)=0\;\forall x\notin X\\
 & \mu\geq0
\end{aligned}
\]
In this program, $s$ goes over all possible faces, $T_{X_{s}}\succeq\lambda_{|s|}\mathrm{Id}$
means that $T_{X_{s}}-\lambda_{|s|}\mathrm{Id}$ is positive semidefinite,
and $\mu\geq0$ means that all entries of $\mu$are nonnegative. If
a solution to the program has objective value $\beta$, then this
gives a bound of $1-1/\beta$ on the $\alpha$-measure of an independent
set of $ X$.

Since the maximal eigenvalue of $T_{X_{s}}$ is always $1$, we can
rephrase the constraint $T_{X_{s}}\succeq\lambda_{|s|}\mathrm{Id}$
equivalently as follows: the spectral radius of $\mathrm{Id}-T_{X_{s}}$
is at most $1-\lambda_{|s|}$. Using Schur complements, this is easily
seen to be equivalent to the semidefinite constraint 
\[
\left(\begin{array}{cc}
(1-\lambda_{|s|})\mathrm{Id} & \mathrm{Id}-T_{X_{s}}\\
\mathrm{Id}-T_{X_{s}}^{T} & (1-\lambda_{|s|})\mathrm{Id}
\end{array}\right)\succeq0.
\]
Making $1-\lambda_{|s|}$ a variable, we have expressed the problem
of finding the best generalized Hoffman bound as minimizing a semidefinite
program whose objective value is a product of $k-1$ variables. When
$k=2$, this is just a semidefinite program, which can be solved efficiently;
up to the nonnegativity constraint $\mu\geq0$, we have recovered
the Lovász $\theta$ function. When $k>2$, the objective function
is no longer convex, and it is not clear how to solve the program
efficiently.

\section{Frankl's problem on extended triangles\label{sec:Frankl-triangle}}

\subsection{The uniform version}
Frankl's Turán problem on hypergraphs without extended triangles reads
as follows.  A \emph{triangle} in $\mathcal{P}\left(\left[n\right]\right)$, the power set on $[n]$, is a $2k$-uniform hypergraph supported on three sets $\left\{ A,B,C\right\} $ such that
each element of $\left[n\right]$ belongs to an even number of the
sets $A,B,C.$ In other words, there exist disjoint $k$-element sets
$D,E,F$ such that $D\cup E=A,$ $D\cup F=B,$ and $E\cup F=C.$ Frankl,~\cite{frankl1990asymptotic}, asked how large can a family $\mathcal{F}\subseteq\binom{\left[n\right]}{2k}$
be if it does not contain a triangle. The reason for considering only
even uniformities is that no $k$-uniform triangle exists for an odd
$k$. Equivalently, we are interested in the maximum independent set
in the 3-uniform hypergraph $ X$ whose vertices are the
$2k$-subsets of $[n]$ and whose hyperedges are triangles.

The skeleton of $ X$ is the graph on the same set of vertices
whose edges are pairs of subsets whose intersection has size exactly
$k$. This graph is also known as the generalized Johnson graph $J(n,2k,k)$.
Brouwer et al.,~\cite{brouwer2018smallest}, showed that when $n\le4k-1$,
the minimum eigenvalue is
\[
\lambda_{0}=\frac{n-4k}{2(n-2k)}=1-\frac{n}{2(n-2k)}.
\]\iffalse
\[
\lambda_{0}=\frac{E_{k}(1)}{E_{k}(0)}=\frac{n-4k}{2(n-2k)}=1-\frac{n}{2(n-2k)}.
\]\fi
Since the $3$-faces in $X$ are the triples of the form $A,B,A\triangle B$, the links of a vertex is a perfect matching, hence $\lambda_{1}=-1$. It follows that when $n\le4k-1$, the size of an independent set is at most

\[
\binom{n}{2k}\left(1-\frac{1}{2(1-\lambda_{0})}\right)=\binom{n-1}{2k-1}.
\]
In particular, when $n=4k-1$, an independent set contains at most
a $\frac{2k}{4k-1}=\frac{1}{2}+O(\frac{1}{k})$ fraction of subsets.
Now suppose that $n\ge4k$, and let $\mathcal{F}$ be a triangle-free
family. Consider the following random experiment: choose a random
$(4k-1)$-subset $S$ of $[n]$, and check whether a random $2k$-subset
of $S$ belongs to $\mathcal{F}$. On the one hand, this is at most
$\frac{2k}{4k-1}$. On the other hand, this is equal to the density
of $\mathcal{F}$. This completes the proof of the following theorem.
\begin{thm}
If $\mathcal{F}$ is a family of $2k$-subsets of $[n]$ which does
not contain three distinct subsets whose symmetric difference is empty,
then $|\mathcal{F}|\le\binom{n-1}{2k-1}$ if $n\le4k-1$, and $|\mathcal{F}|\leq(1/2+O(1/k))\binom{n}{2k}$
otherwise. 

This bound is sharp: If $n\le4k-1$ then the family of
$2k$-sets containing a fixed element satisfies the condition and
has size $\binom{n-1}{2k-1}$. Otherwise, the family of $2k$-sets
containing an odd number of elements among the first $\lfloor n/2\rfloor$
satisfies condition and asymptotically contains half the $2k$-sets.
\end{thm}

Frankl,~\cite{frankl1990asymptotic}, gave the upper bound $(1/2+O(1/n))\binom{n}{k}$, which is slightly better for large $n$.

\subsection{The $p$-biased version}

The $p$-biased version of the problem is as follows:
\begin{quotation}
Given $p\le\frac{2}{3}$, how large can $\mu_{p}\left(\mathcal{F}\right)$
be if $\mathcal{F}\subseteq\mathcal{P}\left(\left[n\right]\right)$
does not contain a triangle?
\end{quotation}
The reason for the condition $p\le\frac{2}{3}$ is the fact that
the example $\left\{ A:\,\left|A\right|>\frac{2}{3}n\right\} $ is
triangle-free, and its $p$-biased measure tends to $1$ as $n$ tends
to infinity.

The $p$-biased version of Frankl's problem is the problem of
determining the independence number of the $3$-uniform hypergraph
$X^{\otimes n},$ where $X=\left(V,\mu\right)$ is with $V=\left\{ 0,1\right\} $
and 
\[
\mu\left([1,1,0]\right)=\frac{3}{2}p,\:\mu\left([0,0,0]\right)=1-\frac{3}{2}p.
\]
The induced measures are 
\begin{gather*}
\mu_{2}([1,1])=\frac{1}{2}p,\:\mu_{2}\left([1,0]\right)=p,\:\mu_{2}\left([0,0]\right)=1-\frac{3}{2}p;\\
\mu_{1}(0)=1-p,\:\mu_{1}(1)=p.
\end{gather*}
And therefore, the matrix of the normalized adjacency operator $T_{X}$
on the skeleton of $X$ is
\[
T_{X}=\left(\begin{array}{cc}
\frac{1-\frac{3}{2}p}{1-p} & \frac{\frac{1}{2}p}{1-p}\\
\frac{\frac{1}{2}p}{p} & \frac{\frac{1}{2}p}{p}
\end{array}\right)=\left(\begin{array}{cc}
\frac{2-3p}{2(1-p)} & \frac{p}{2(1-p)}\\
\frac{1}{2} & \frac{1}{2}
\end{array}\right),
\]
with eigenvalues $1$ and $\frac{1-2p}{2(1-p)}$. The induced distribution
on the link of $[0]$ is supported on the faces $[0,0]$ and $[1,1]$,
hence the corresponding matrix $T_{X_{0}}$ is
\[
T_{X_{0}}=\left(\begin{array}{cc}
1 & 0\\
0 & 1
\end{array}\right),
\]
while the link of $[1]$ is supported on $[0,1]$, and $T_{X_{1}}$
is
\[
T_{X_{1}}=\left(\begin{array}{cc}
0 & 1\\
1 & 0
\end{array}\right).
\]
The above shows that $\lambda_{0}=\frac{1-2p}{2(1-p)}$ and $\lambda_{1}=-1$.
When $p > 1/2$, $\lambda_{0}$ is negative, implying
\[
\alpha(X^{\otimes n})\leq1-\frac{1}{(1-\frac{1-2p}{2(1-p)})\cdot2}=p.
\]
When $p\le1/2$, $\lambda_{0}(X)$ is nonnegative, and so $\lambda_{0}(X^{\otimes n})\ge0$,
implying

\[
\alpha(X^{\otimes n})\leq1-\frac{1}{1\cdot2}=\frac{1}{2}.
\]
This completes the proof of the following statement.
\begin{thm}
Let $\{0,1\}^{n}$ denote the space of $\{0,1\}$-vectors of length
$n$, and $\mu$ be the $p$-biased measure on it, with $p\leq2/3$.
If $\mathcal{F}\subseteq\{0,1\}^{n}$ is a family of vectors which
does not contain three distinct vectors whose sum to zero, then $\mu(\mathcal{F})\leq\max(p,1/2)$.

This bound is sharp: if $p\leq1/2$ then the family of vectors having
odd parity satisfies the condition and has measure tending to $1/2$
as $n\to\infty$, and if $p\ge1/2$ then the set of all vectors having
$1$ as their first coordinate satisfies the condition and has measure
$p$.
\end{thm}
\section{Mantel's Theorem\label{sec:Mantel-proof}}
The classical Mantel's theorem bounds the number of edges in a triangle-free
graph.
\begin{thm}
\cite{mantel1907} If a graph on $n$ vertices contains no triangles,
then it contains at most $\left\lfloor \frac{n^{2}}{4}\right\rfloor $
edges.
\end{thm}
We give a spectral proof of Mantel's Theorem for graphs with $2^{n}$
vertices that relies on a variation of the bound (\ref{eq:generalized Hoffman bound}).
Apart from presenting a spectral proof of this, we would like to show
the flexibility of the presented spectral approach. In some cases, one can improve the bound (\ref{eq:generalized Hoffman bound})
by taking not the smallest eigenvalue of the normalized adjacency
operator but a larger one. This is possible when the characteristic
function of the independent set we are interested in is orthogonal
to the eigenfunctions that correspond to smaller eigenvalues.
\begin{proof}
First, we encode the statement of the theorem in terms of the independent
sets of a hypergraph. Let $G$ be a triangle-free graph on $2^{n}$
vertices, which we identify with the set $[2^{n}]$. Let $\mathcal{X}_{2^{n}}$
be a $3$-partite $3$-uniform hypergraph on the vertex set $V=V_{1}\cup V_{2}\cup V_{3}$,
where each part $V_{i},\:i=1,2,3,$ is a copy of the set $\left[2^{n}\right]\times\left[2^{n}\right]$.
The $3$-faces of $\mathcal{X}_{2^{n}}$ are the triples of the form
$\left[\left(i,j\right),\left(j,k\right),\left(k,i\right)\right]$,
where $1\leq i,j,k\leq2^{n}$. Assume the probability distribution
on the $3$-faces to be the uniform probability distribution on these
triples. We encode $G$ as an independent set $I$ of $\mathcal{X}_{2^{n}}$
, namely, $I=I_{1}\cup I_{2}\cup I_{3}$, where for each $i=1,2,3$,
\[
I_{i}=\left\{ \left(v,u\right)\in V_{i}:\,\text{the set }\left\{ v,u\right\} \text{ makes an edge of }G\right\} .
\]
Note that $\mathcal{X}_{2}^{\otimes n}=\mathcal{X}_{2^{n}}$ and hence
this gives the desired encoding of $G$ as the independent set $I$
in a $3$-uniform hypergraph which is also a tensor power.

It follows immediately from the construction of $\mathcal{X}_{2^{n}}$,
in particular, from the fact that it is $3$-partite, that the matrix
of the normalized adjacency operator $T$ on the skeleton of $\mathcal{X}_{2^{n}}$
is the Kronecker product of the following $3\times3$ matrix
\[
M=\left(\begin{array}{ccc}
0 & \frac{1}{2} & \frac{1}{2}\\
\frac{1}{2} & 0 & \frac{1}{2}\\
\frac{1}{2} & \frac{1}{2} & 0
\end{array}\right)
\]
and the matrix $M_{n}$ which, in turn, is the $n$-th tensor power
of $M_{1}$, given by the following $4\times4$ matrix indexed by
the elements of $[2]\times[2]$:
\[
\left.M_{1}=\begin{array}{c}
\left(1,1\right)\\
\left(1,2\right)\\
\left(2,1\right)\\
\left(2,2\right)
\end{array}\underset{\begin{array}{cccc}
\left(1,1\right) & \left(1,2\right) & \left(2,1\right) & \left(2,2\right)\end{array}}{\begin{cases}
\underbrace{\left(\begin{array}{cccc}
\frac{1}{2} & \frac{1}{4} & \frac{1}{4} & 0\\
\frac{1}{4} & 0 & \frac{1}{2} & \frac{1}{4}\\
\frac{1}{4} & \frac{1}{2} & 0 & \frac{1}{4}\\
0 & \frac{1}{4} & \frac{1}{4} & \frac{1}{2}
\end{array}\right)}\end{cases}}.\right.
\]
Since $T$ is the Kronecker product of $M$ and $M_{n}$, its eigenvalues
are exactly the products of the eigenvalues of $M$ and $M_{n}$.
The eigenvectors and eigenvalues of $M$ are 
\[
\left(\left(\begin{array}{c}
1\\
1\\
1
\end{array}\right),1\right),\left(\left(\begin{array}{c}
1\\
0\\
-1
\end{array}\right),-\frac{1}{2}\right),\left(\left(\begin{array}{c}
1\\
-1\\
0
\end{array}\right),-\frac{1}{2}\right).
\]
The eigenvectors and eigenvalues of $M_{1}$ are
\begin{multline*}
\left(\chi_{1},\lambda_{1}\right)=\left(\left(\begin{array}{c}
1\\
1\\
1\\
1
\end{array}\right),1\right),\left(\chi_{2},\lambda_{2}\right)=\left(\left(\begin{array}{c}
-1\\
1\\
1\\
-1
\end{array}\right),0\right),\\
\left(\chi_{3},\lambda_{3}\right)=\left(\left(\begin{array}{c}
1\\
0\\
0\\
-1
\end{array}\right),\frac{1}{2}\right),\left(\chi_{4},\lambda_{4}\right)=\left(\left(\begin{array}{c}
0\\
1\\
-1\\
0
\end{array}\right),-\frac{1}{2}\right).
\end{multline*}
We now exploit the symmetries of the set $I$ to show that its characteristic
function $1_{I}$ is orthogonal to the subspace of eigenvectors of
$T$ with negative eigenvalues. First, note that the set $I$ is invariant
under the action of the symmetric group $S_{3}$ acting on $\mathcal{X}_{2^{n}}$
by permutations of the parts $\left\{ V_{1},V_{2},V_{3}\right\} $.
The only eigenvector of $M$ invariant under the action of $S_{3}$
is the constant vector with eigenvalue $1$. A vertex in $\mathcal{X}_{2}$
is a pair $\left(i,j\right)$. Let $\mathcal{S}_{0}$ be the operator
that swaps between $i$ and $j$, which satisfies 
\[
\mathcal{S}_{0}\chi_{i}=\begin{cases}
\chi_{i} & i=1,2,3\\
-\chi_{i} & i=4
\end{cases}.
\]
The operator $\mathcal{S}$ that swaps the coordinates of a vertex
in $\mathcal{X}_{2^{n}}$ is of the form $\mathcal{S}=\mathcal{S}_{0}^{\otimes n}$.
The set $I$ is invariant under the action of $\mathcal{S}$ by the
construction. 

Since the characteristic function $1_{I}$ is orthogonal to the subspace
spanned by eigenvectors of the normalized adjacency operator with
negative eigenvalues, we may take $0$ instead of $\lambda_{0}$ in
(\ref{eq:generalized Hoffman bound}). Note that since the link of
most vertices in $\mathcal{X}_{2^{n}}$ is bipartite, $\lambda_{1}=-1$
(which is the minimal possible eigenvalue of a Markov matrix). Hence,
the bound (\ref{eq:generalized Hoffman bound}) reads as
\[
\frac{|I|}{3\cdot4^{n}}\leq1-\frac{1}{1\cdot2}=\frac{1}{2}.
\]
Taking into account the fact that every edge of $G$ is counted six
times in $|I|$ completes the proof.
\end{proof}

\section{Frankl-Tokushige Theorem on Intersecting Families}
Our method also provides a new proof for the result of Frankl and Tokushige on $k$-wise intersecting families,~\cite{frankl2003weighted}.
\begin{thm}\cite{frankl2003weighted}
Let $k\ge2$ and $p\leq1-\frac{1}{k}$. Assume $\mathcal{F}\subset\mathcal{P}([n])$ is $k$-wise intersecting, that is, for all $F_1,\dots,F_k\in\mathcal{F}$
\[
 F_1\cap\dots\cap F_k \neq \emptyset.
\]
Then $\mu_p(\mathcal{F}) \leq p$, where $\mu_p$ stands the $p$-biased measure.
\end{thm}
\begin{proof}
Let $X$ be the $k$-uniform hypergraph on $\{0,1\}$ weighted by the measure $\mu([0^{(k)}])=1-\frac{k}{k-1}p$, $\mu([0,1^{(k-1)}])=\frac{k}{k-1}p$. Here $0^{(k)}$ means $k$ copies of $0$. The induced measure $\mu_1$ on the vertex set $\{0,1\}$ is the $p$-biased one, i.e., $\mu_1(1) = p$ and $\mu_1(0)=1-p$. The matrix form of $T_{X}$ is directly calculated to be
\[
\begin{pmatrix}\frac{1-\frac{k}{k-1}p}{1-p} & \frac{\frac{1}{k-1}p}{1-p}\\
\frac{1}{k-1} & \frac{k-2}{k-1}
\end{pmatrix},
\]
from which it follows that its smallest eigenvalue is $\lambda_{0} = \frac{\frac{k-2}{k-1}-p}{1-p}<0$ hence $\frac{1}{1-\lambda_{0}}=(k-1)(1-p)$.

In order to calculate $\lambda_{i}$ for $i>0$, notice that the only
links with non-trivial $\mu_{1}(X_{S})$ are of the faces $S=[1^{\ell}]$.
The matrix form of $T_{X_{S}}$ is directly calculated to be
\[
\begin{pmatrix}0 & 1\\
\frac{1}{k-1-\ell} & \frac{k-2-\ell}{k-1-\ell}
\end{pmatrix},
\]
and so $\lambda_{\ell}=-\frac{1}{k-1-\ell}<0$ and $\frac{1}{1-\lambda_{\ell}}=\frac{k-1-\ell}{k-\ell}$.
Applying the generalized Hoffman bound for tensor powers, we conclude
\[
\alpha(X^{\otimes n})\leq1-(k-1)(1-p)\cdot\frac{k-2}{k-1}\cdot\cdots\cdot\frac{1}{2}=p.
\]
Since the edges of $X$ are exactly the multisets with either all 0's or
exactly one 0, every $k$-wise intersecting set is independent in $X^{\otimes n}$.
\end{proof}

\iffalse

\section{Open problems\label{sec:Open-problems}}
\begin{itemize}
\item Stability?
\item $s$-wise $t$-intersecting problem may be solved using this method?
\item Which Turán problems for hypergraphs may be solved using this method?
\end{itemize}

\section{TODO's}
\begin{enumerate}
\item Noam: state and prove k-uniform version of the Frankl's triangle problem
(Yuval: Done)
\item Yuval: add a section with optimization approach (Done)
\item Kosta: fix the bibliography
\item Yuval \& Noam: read the whole thing (Yuval: Done)
\item Yuval \& Noam: write open problems and further directions
\end{enumerate}
%
\fi

\section*{Acknowledgements}
The authors are grateful to Ehud Friedgut, Gil Kalai, Guy Kindler, and Dor Minzer for valuable discussions.

\subsection*{Funding}
YF is a Taub Fellow, and is supported by the Taube Foundation and ISF grant 1337/16.
KG is currently supported by the SNF grant number 20002\_169106, and by ERC grant 336283 while at the Weizmann Institute and Bar-Ilan University.

\bibliographystyle{plain}
\bibliography{Refs}

\end{document}